\documentclass[11pt]{article}
\usepackage[T1]{fontenc}
\usepackage{lmodern}
\usepackage{amsmath,amssymb,amsthm,mathtools}
\usepackage[a4paper,margin=29mm]{geometry}
\usepackage{microtype}
\usepackage[hidelinks]{hyperref}
\numberwithin{equation}{section}
\newtheorem{theorem}{Theorem}[section]
\newtheorem{lemma}[theorem]{Lemma}
\newtheorem{proposition}[theorem]{Proposition}

\theoremstyle{definition}
\newtheorem{definition}[theorem]{Definition}
\theoremstyle{remark}

\newcommand{\Vsub}{\mathcal V_{\leq 1}}
\newcommand{\Vone}{\mathcal V_1}
\newcommand{\OO}{\mathcal O}
\newcommand{\id}{\mathrm{id}}
\newcommand{\up}{\mathord{\uparrow}}

\newcommand{\pos}[1]{(#1)_+}

\title{Probabilistic Powerdomains of Pointed FS-Domains
Are FS-Domains}
\author{Xiaoyong Xi \and Chong Shen \and Dongsheng Zhao}
\date{ }

\begin{document}
\maketitle

\begin{abstract}
The Jung--Tix problem concerns the compatibility of probabilistic powerdomains
with Cartesian closed categories of domains. We prove that both the
subprobability and probability valuation powerdomains of every pointed
FS-domain are FS-domains.
\end{abstract}

{\small
\par
}

\section{Introduction}

The Jung--Tix problem asks for compatibility between probabilistic
powerdomains and Cartesian closed categories of semantic domains
\cite{JungTix}. We address its FS-domain closure form: whether the
subprobability valuation powerdomain of every pointed FS-domain is FS.
This is distinct from finding one Cartesian closed class of continuous
domains preserved by valuation powerdomains.

Chen, Kou and Lyu establish such closure for the countably based class
$\omega\mathbf{FVA}$, and extend their construction to the class
$\mathbf{FVA}$ without a countability restriction
\cite[Theorem~5.22 and Remark~5.23]{ChenKouLyu}. Their conclusion leaves
the general FS closure question open. The scope of the present theorem
is every pointed FS-domain, without requiring finite-valuation
factorization approximants.

For a dcpo $D$, we write $\Vsub(D)$ for the dcpo of continuous
subprobability valuations on $D$, and $\Vone(D)$ for the dcpo of
continuous probability valuations on $D$. The subscripts specify the
total mass: $\mu(D)\leq1$ in $\Vsub(D)$ and $\mu(D)=1$ in $\Vone(D)$.
Both are ordered pointwise on Scott-open sets: $\mu\leq\nu$ means
$\mu(U)\leq\nu(U)$ for every Scott-open $U\subseteq D$.
The word \emph{pointed} means that $D$ has a least element, denoted by
$\bot$. The definitions of FS-domains and continuous valuations are
recalled in Section~\ref{sec:prelim}.

\begin{theorem}\label{thm:main}
If $D$ is a pointed FS-domain, then both $\Vsub(D)$ and
$\Vone(D)$ are FS-domains.
\end{theorem}

The essential construction is nonlinear. It retains variable valuation
components and adds a finite collection of anchor contributions.
Those contributions share one capacity budget. Their increments
compensate for the decrease of a raw coefficient under an upward
transport. Finite valuations supported on these anchors provide the
separating witnesses; the outputs themselves need not be supported on
those anchors.

The proof uses the continuity and simple-basis theorem for valuations
\cite[Theorem~5.2 and Corollary~5.4]{Jones}, together with finite transport.
The construction specific to this proof is the integrated anchor budget.
The square-refinement
argument is related to \cite[Lemma~4.4]{ChenKouLyu}; we state and prove
the finite-batch version needed here in Lemma~\ref{lem:saturation}.

Section~\ref{sec:prelim} recalls the valuation facts and proves the
square-saturation lemma used below. In Section~\ref{sec:construction},
we construct the approximating maps and establish their Scott continuity
and finite separation. Finally, Section~\ref{sec:cofinal} combines these
maps into a directed FS approximate identity and completes the proof
of Theorem~\ref{thm:main}.

\section{Preliminaries}\label{sec:prelim}

A dcpo is ordered by $\leq$, and directed sets are nonempty.
A subset is directed if every pair of its elements has a common upper
bound in that subset; consequently every finite nonempty subset does too.
We write $\sup A$ for the least upper bound of $A$. The way-below relation
$u\ll x$ means
\[
 \text{for every directed }A\subseteq X,\qquad
 x\leq\sup A\ \Longrightarrow\ (\exists a\in A)\ u\leq a.
\]
Thus $u\ll x$ says that whenever a directed family reaches or exceeds $x$
at its supremum, one member of that family already lies above $u$.
A domain is a continuous dcpo.
A pointed dcpo has a least element, denoted by $\bot$.
We write $\up x=\{y:x\leq y\}$ for the principal upper set of $x$.
The notation $x_i\uparrow x$ means that $(x_i)$ is an increasing
directed family with supremum $x$.
Maps with the same domain and codomain are ordered pointwise:
$p\leq q$ means $p(x)\leq q(x)$ for every $x$.
Juxtaposition of maps denotes composition, so $pq=p\circ q$ and
$q^2=q\circ q$; $\id_X$ denotes the identity map on $X$.
Maps between dcpos are Scott continuous when they are monotone and
preserve directed suprema.
All convergence statements indexed by a directed set refer to nets;
no countable cofinal subset is assumed.

\begin{definition}
A Scott-continuous map $q:X\to X$ is \emph{finitely separated} if a
finite set $M\subseteq X$ satisfies
\begin{equation}\label{eq:fsdef}
 \forall x\in X\quad \exists m\in M\quad q(x)\leq m\leq x.
\end{equation}
Such a set $M$ is called a finite separator for $q$; its elements
will also be called separating labels.
A dcpo $X$ is an FS-domain if it has a directed family of such maps
whose pointwise supremum is $\id_X$.
\end{definition}

We use two standard facts: every finitely separated map $q:X\to X$
satisfies $q(x)\ll x$ for every $x\in X$, and every FS-domain is a
continuous dcpo, hence a domain; see \cite[Lemma~2.5]{ChenKouLyu}.

Write $\OO(D)$ for the collection of Scott-open subsets of $D$.
A continuous subprobability valuation on a dcpo $D$ is a strict,
monotone, modular map $\mu:\OO(D)\to[0,1]$ preserving directed unions
of Scott-open sets. Here strict means $\mu(\varnothing)=0$.
Modularity means
\[
 \mu(U)+\mu(V)=\mu(U\cup V)+\mu(U\cap V)
 \qquad(U,V\in\OO(D)).
\]
The order on $E=\Vsub(D)$ is
\[
 \mu\leq\nu\quad\Longleftrightarrow\quad
 \mu(U)\leq\nu(U)\quad(U\in\OO(D)).
\]
Directed suprema in $E$ are computed on opens. For a Scott-continuous
map $g:D\to D$, the pushforward operator $g_*:E\to E$ is defined by
$(g_*\mu)(U)=\mu(g^{-1}(U))$. Thus $g_*\mu$ is the valuation obtained
by sending the mass through $g$.
The probability valuation dcpo is
$\Vone(D)=\{\mu\in E:\mu(D)=1\}$, with the inherited order.
Its directed suprema are also computed on opens.
A simple valuation is a finite sum $\sum_j p_j\delta_{x_j}$ with
$p_j\geq0$ and $\sum_jp_j\leq1$.
Here $\delta_x$ is the Dirac valuation at $x$: for a Scott-open set $U$,
$\delta_x(U)=1$ if $x\in U$, and $\delta_x(U)=0$ otherwise.
The coefficient $p_j$ is the mass placed at $x_j$, and $\mu(D)$ is the
total mass of $\mu$. Empty sums denote the zero valuation.

We use the standard valuation-domain theorem: if $D$ is continuous,
then $E$ is continuous, and, for each $\mu$, the simple valuations
way below $\mu$ form a directed basis with supremum $\mu$
\cite[Theorem~8]{JungTix}; see also
\cite[Theorem~5.2 and Corollary~5.4]{Jones}.

The following elementary-transport consequence of the splitting lemma
\cite[Lemma~5]{JungTix} extends the finite-domain formulation in
\cite[Lemma~6]{JungTix} to simple valuations on a continuous dcpo.
We give the proof and use it in Lemma~\ref{lem:simplemono}.

\begin{lemma}[Elementary transport]\label{lem:finiteflow}
Let $D$ be a continuous dcpo and let $\mu,\nu\in\Vsub(D)$ be simple
valuations. Then
\[
 \mu\leq\nu
 \quad\Longleftrightarrow\quad
 \begin{gathered}
  \exists N\in\mathbb N_0\ \exists\text{ simple }
  \lambda_0,\ldots,\lambda_N\in\Vsub(D):\\
  \lambda_0=\mu,\quad\lambda_N=\nu,\quad
  \text{each step has one of the following two forms.}
 \end{gathered}
\]
For $0\leq k<N$, the permitted forms are
\[
 \begin{aligned}
 \text{(move)}\quad&
  \lambda_k=\rho_k+c_k\delta_{x_k},\qquad
  \lambda_{k+1}=\rho_k+c_k\delta_{y_k},
  &&x_k\leq y_k,\quad c_k\geq0;\\
 \text{(add)}\quad&
  \lambda_{k+1}=\lambda_k+c_k\delta_{y_k},
  &&c_k\geq0.
 \end{aligned}
\]
Here $\mathbb N_0=\{0,1,2,\ldots\}$, $x_k,y_k\in D$, and $\rho_k$ is
a simple valuation with nonnegative coefficients. In a move, $c_k$ units
of mass already present at $x_k$ are moved to $y_k$, while the remainder
$\rho_k$ stays fixed. An addition places $c_k$ new units at $y_k$.
Zero-mass steps may be omitted. The chain satisfies
\[
 \mu=\lambda_0\leq\lambda_1\leq\cdots\leq\lambda_N=\nu,
 \qquad
 \mu(D)\leq\lambda_k(D)\leq\nu(D)\leq1.
\]
\end{lemma}
\begin{proof}
Suppose first that $\mu\leq\nu$. Write
\[
 \mu=\sum_{i=1}^{r}\alpha_i\delta_{x_i},\qquad
 \nu=\sum_{j=1}^{s}\beta_j\delta_{y_j},
 \qquad \alpha_i,\beta_j>0,
\]
with distinct points within each support. The supports on the two sides
may overlap. If $\mu=0$, successive additions of the target masses give
the required chain. If $\nu=0$, the inequality $\mu\leq\nu$ forces
$\mu=0$, so the chain of length $N=0$ suffices. We may therefore assume
that $r,s\geq1$.

\emph{Step 1: split the source mass among permissible targets.}
The splitting lemma \cite[Lemma~5]{JungTix} supplies nonnegative real
numbers $t_{ij}$ satisfying
\[
 \sum_{j=1}^{s}t_{ij}=\alpha_i\quad(1\leq i\leq r),\qquad
 \sum_{i=1}^{r}t_{ij}\leq\beta_j\quad(1\leq j\leq s),\qquad
 t_{ij}>0\ \Longrightarrow\ x_i\leq y_j.
\]
The entry $t_{ij}$ is the amount of source mass sent from $x_i$ to $y_j$.
The row equality says that all mass at each source is assigned; the
column inequality says that no target receives more than its prescribed
mass. The last condition permits transport only upward in the order.

\emph{Step 2: perform these finitely many moves.}
List the pairs with $t_{ij}>0$ as
$(i_1,j_1),\ldots,(i_L,j_L)$, and put $c_\ell=t_{i_\ell j_\ell}$.
After the first $\ell$ moves, use the valuation
\[
 \lambda_\ell=
 \sum_{h=1}^{\ell}c_h\delta_{y_{j_h}}
 +\sum_{h=\ell+1}^{L}c_h\delta_{x_{i_h}}
 \qquad(0\leq\ell\leq L).
\]
The first sum consists of portions already moved to their targets; the
second consists of portions still at their sources. The row equalities
give $\lambda_0=\mu$. For $1\leq\ell\leq L$, the transition from
$\lambda_{\ell-1}$ to $\lambda_\ell$ moves the portion $c_\ell$ still
present at $x_{i_\ell}$ to $y_{j_\ell}$. This is a permitted move because
$x_{i_\ell}\leq y_{j_\ell}$. All other terms form a nonnegative
remainder $\rho$. This formula also shows that a portion already moved
is never required to move again, even when source and target points
coincide. Every such step preserves total mass, hence
$\lambda_\ell(D)=\mu(D)\leq1$.

\emph{Step 3: add the unfilled target masses.}
For each target put
\[
 d_j=\beta_j-\sum_{i=1}^{r}t_{ij}\geq0.
\]
At the end of Step 2 we have
$\lambda_L=\sum_j(\beta_j-d_j)\delta_{y_j}$.
Define the remaining steps by
\[
 \lambda_{L+j}=\lambda_L+\sum_{h=1}^{j}d_h\delta_{y_h}
 \qquad(1\leq j\leq s).
\]
Thus $\lambda_{L+s}=\nu$, and each new step is an addition.
During these additions the total mass is
$\mu(D)+\sum_{h=1}^{j}d_h\leq\nu(D)\leq1$.
This constructs the required finite chain.

Conversely, each permitted step increases the valuation order.
Indeed, Scott-open sets are upper sets, so $x\leq y$ implies
$\delta_x(U)\leq\delta_y(U)$ for every Scott-open $U$.
Consequently a move from $x$ to $y$ cannot decrease any open-set
evaluation. An addition also cannot decrease such an evaluation,
because its coefficient is nonnegative. Therefore
$\lambda_k\leq\lambda_{k+1}$ at every step, and transitivity gives
$\mu\leq\nu$. Evaluating this chain on $D$ gives the stated total-mass
bounds.
\end{proof}

We next use a finite-batch form of the square-refinement argument in
\cite[Lemma~4.4]{ChenKouLyu}. It constructs a directed FS approximate
identity from the simultaneous approximation condition stated below.

\begin{lemma}[Finite-batch square saturation]\label{lem:saturation}
Let $X$ be a continuous dcpo, and let $\mathcal H$ be a nonempty set
of finitely separated maps on $X$. Suppose that for every finite
list $u_j\ll x_j$ there is $h\in\mathcal H$ with
$u_j\leq h(x_j)$ for all $j$. Then
$\{h^2:h\in\mathcal H\}$ is a directed FS approximate identity on $X$.
\end{lemma}
\begin{proof}
Take $p,q\in\mathcal H$. For each $\ell\in\{p,q\}$, choose a finite
separator $M_\ell$. For each $m\in M_\ell$, the standard
finite-separation property \cite[Lemma~2.5(1)]{ChenKouLyu} gives
$\ell(m)\ll m$. By interpolation, choose $z_{\ell,m}$ with
$\ell(m)\ll z_{\ell,m}\ll m$.
Apply the finite-batch assumption to both pairs
$(\ell(m),z_{\ell,m})$ and $(z_{\ell,m},m)$, for every
$\ell\in\{p,q\}$ and $m\in M_\ell$. It supplies one $h\in\mathcal H$
satisfying all the inequalities
\[
 \ell(m)\leq h(z_{\ell,m}),\qquad z_{\ell,m}\leq h(m).
\]
For any $x\in X$ and either $\ell=p$ or $\ell=q$, choose
$m\in M_\ell$ with $\ell(x)\leq m\leq x$. Then
\[
 \ell^2(x)\leq\ell(m)\leq h(z_{\ell,m})\leq h^2(m)\leq h^2(x).
\]
Consequently $p^2,q^2\leq h^2$, proving directedness.
Each $h^2$ is finitely separated by the labels for $h$, since
$h^2\leq h\leq\id_X$.

Given $u\ll x$, choose $u\ll z\ll x$ and one $h$ such that
$u\leq h(z)$ and $z\leq h(x)$. Then $u\leq h^2(x)$.
Continuity of $X$ gives $\sup_{h\in\mathcal H}h^2(x)=x$.
\end{proof}

\section{Construction of finitely separated approximants}\label{sec:construction}

We first organize the underlying domain into finitely many profiles and
choose an anchor for each profile. These anchors will support both the
compensating masses and the finite separating witnesses.
Throughout the construction, $D$ is a pointed continuous dcpo.
Fix a finitely separated map $f:D\to D$ and a nonempty finite
separator $M$. Define
\begin{equation}\label{eq:profile}
 k(x)=\{m\in M:f(x)\leq m\},\qquad R=k[D],\qquad n=|R|.
\end{equation}
Here $R=k[D]=\{k(x):x\in D\}$ is the set of profiles and $n=|R|$
is its cardinality. The finite poset $R$ is ordered by reverse inclusion:
\[
 S\leq T\text{ in }R\quad\Longleftrightarrow\quad S\supseteq T.
\]
Thus $\up S=\{T\in R:T\subseteq S\}$, and $T>S$ means
$T\geq S$ with $T\ne S$. In particular
$k$ is monotone. It is Scott continuous: if $x_i\uparrow x$, then
$f(x_i)\uparrow f(x)$; each $m\notin k(x)$ is eventually absent
from $k(x_i)$, and finitely many eventual absences hold simultaneously.
Hence $k(x_i)=k(x)$ eventually. We have $n\geq1$.

For $S\in R$, put
\[
 C_S=k^{-1}\{S\},\quad A_S=k^{-1}(\up S),\quad
 B_S=k^{-1}(\up S\setminus\{S\}).
\]
Both $A_S$ and $B_S$ are Scott open, and $C_S=A_S\setminus B_S$.
Choose once and for all $x_S\in C_S$, and set
\begin{equation}\label{eq:anchors}
 g=f^2,\qquad h_S=f(x_S).
\end{equation}

\begin{lemma}\label{lem:anchor}
The anchors satisfy
\begin{align}
 g(x)&\leq h_S &&(x\in C_S),\label{eq:anchorupper}\\
 h_S&\leq y &&(k(y)\geq S).\label{eq:anchorlower}
\end{align}
\end{lemma}
\begin{proof}
For~\eqref{eq:anchorupper}, choose a separating label $m$ for $x_S$.
Then $m\in S=k(x)$ and $m\leq x_S$, so
$f^2(x)\leq f(m)\leq f(x_S)=h_S$.
For~\eqref{eq:anchorlower}, choose a separating label $m$ for $y$.
Since $k(y)\subseteq S$, we have $m\in S$, and therefore
$h_S=f(x_S)\leq m\leq y$.
\end{proof}

The assignment $S\mapsto h_S$ need not be monotone. The proof only
uses the two displayed inequalities, in their stated directions.

\begin{lemma}\label{lem:cell}
For $\mu\in E$ the formula
\begin{equation}\label{eq:restriction}
 \mu_S(U)=\mu(U\cap A_S)-\mu(U\cap B_S)
 \qquad(U\in\OO(D))
\end{equation}
defines a positive continuous valuation. The cell $C_S$ need not be
Scott open; this formula defines its valuation component using the
two open sets $A_S$ and $B_S$, without evaluating $\mu$ directly on $C_S$.
Moreover
\begin{equation}\label{eq:decomp}
 \sum_{S\in R}\mu_S=\mu,\qquad
 \beta_S=\mu_S(D),\qquad
 r_S=\mu(A_S)=\sum_{T\geq S}\beta_T.
\end{equation}
Here $\beta_S$ is the mass of the component $\mu_S$, whereas $r_S$
is the cumulative mass over all profiles $T\geq S$.
We write $\beta_S(\mu)$ and $r_S(\mu)$ when the input valuation
needs to be displayed; otherwise the argument $\mu$ is suppressed.
If $C_S$ misses an open $U$, then $\mu_S(U)=0$; if $C_S\subseteq U$,
then $\mu_S(U)=\beta_S$.
For a family $(\mu_i)$ we abbreviate $(\mu_i)_S$ by $\mu_{i,S}$ and
write $\beta_{i,S}=\mu_{i,S}(D)$.
Along any directed $\mu_i\uparrow\mu$, $\mu_{i,S}(U)$ converges in
ordinary real numbers to $\mu_S(U)$ for each open $U$.
\end{lemma}
\begin{proof}
Strictness follows from $\mu_S(\varnothing)=0$. Since $B_S\subseteq A_S$,
nonnegativity follows from monotonicity of $\mu$; modularity follows by
subtracting the two finite modular identities in~\eqref{eq:restriction}.
For Scott-open $U\subseteq V$, modularity and monotonicity of $\mu$ give
\[
\begin{aligned}
 \mu(U\cap A_S)+\mu(V\cap B_S)
 &=\mu\bigl((U\cap A_S)\cup(V\cap B_S)\bigr)+\mu(U\cap B_S)\\
 &\leq\mu(V\cap A_S)+\mu(U\cap B_S).
\end{aligned}
\]
Subtracting the $B_S$ terms proves $\mu_S(U)\leq\mu_S(V)$.
For a directed union of opens, both terms in~\eqref{eq:restriction}
converge to their respective finite real limits, so their difference
converges too. Together with the monotonicity just proved, this gives
preservation of directed unions, hence continuity.

For fixed $U$, put $\varphi_U(W)=\mu(U\cap k^{-1}W)$ for upper
subsets $W$ of $R$. This is a valuation on the finite poset $R$.
Its atomic decomposition follows by induction on $|W|$.
For nonempty $W$, choose a minimal element $S$ of $W$.
Then $W\setminus\{S\}$ is upper, and modularity gives
\[
 \varphi_U(W)-\varphi_U(W\setminus\{S\})
 =\varphi_U(\up S)-\varphi_U(\up S\setminus\{S\})
 =\mu_S(U).
\]
Starting with the empty set, induction therefore yields
$\varphi_U(W)=\sum_{S\in W}\mu_S(U)$ for every upper $W$.
Taking $W=R$, and then taking $U=D$ and $W=\up S$, proves
\eqref{eq:decomp}.
If $C_S\cap U=\varnothing$, the two intersections in
\eqref{eq:restriction} coincide. If $C_S\subseteq U$, then
$A_S=B_S\cup(A_S\cap U)$, and modularity gives
$\mu_S(U)=\mu(A_S)-\mu(B_S)=\beta_S$.
The last assertion again follows by taking the difference of two
convergent finite open evaluations, this time along $\mu_i$.
\end{proof}

In particular $r_S\geq r_T$ whenever $S\leq T$. It is useful to
distinguish two kinds of monotonicity. For fixed $\mu$, the map
$U\mapsto\mu_S(U)$ is monotone, as just proved. For fixed $S$ and $U$,
however, the map $\mu\mapsto\mu_S(U)$ need not be monotone: it is the
difference of two open evaluations that may increase by different amounts.
Lemma~\ref{lem:cell} asserts ordinary real convergence of these differences
along directed input nets; it does not assert that they increase.

We next introduce the weights and the additional mass placed at the
anchors. This mass will compensate for coefficient losses when mass
moves to a higher profile.

Keep $f$, its separator, and the resulting profiles and anchors fixed.
Choose $b>0$ such that $\eta=n^2b<1$, and let $a=\exp(-1/b)$.
The numbers $\eta$ and $a$ therefore depend on $b$; $n$ is already fixed.
For $r,t\in[0,1]$ define
\begin{equation}\label{eq:wF}
\begin{aligned}
 w_b(0)&=0,\qquad w_b(t)=\max\{0,1+b\log t\}\quad(t>0),\\
 F_b(r,t)&=\int_0^t\pos{w_b(r)-w_b(s)}\,ds.
\end{aligned}
\end{equation}
The subscript $b$ records the parameter in these functions, and $\log$
denotes the natural logarithm. While $b$ is fixed, we abbreviate
$w_b$ and $F_b$ by $w$ and $F$. We restore the subscript when $b$ varies
in Section~\ref{sec:cofinal}.
Here $u_+=\max\{u,0\}$ denotes the positive part of a real number $u$:
negative values are replaced by zero, and nonnegative values are unchanged.
The functions $w$ and $F$ are continuous; $F$ is nondecreasing in
each argument. Also $0\leq w\leq1$, and $w(r)>0$ precisely when $r>a$.

\begin{lemma}\label{lem:integral}
For $r,t\in[0,1]$,
\begin{equation}\label{eq:Fbound}
 0\leq F(r,t)\leq b(r-a)_+\leq br.
\end{equation}
If $0\leq t\leq t+p\leq r$, then
\begin{equation}\label{eq:increment}
 F(r,t+p)-F(r,t)\geq p\bigl(w(r)-w(t+p)\bigr).
\end{equation}
\end{lemma}
\begin{proof}
The integrand vanishes beyond $r$. Direct integration gives
$F(r,r)=b(r-a)_+$: it is zero for $r\leq a$, and, for $r>a$,
we have $w(s)=0$ on $[0,a]$ and $w(r)=b\log(r/a)$.
On $[a,r]$, the difference is $w(r)-w(s)=b\log(r/s)$.
Splitting the integral at $a$ therefore gives
\[
 F(r,r)=a b\log(r/a)+\int_a^r b\log(r/s)\,ds=b(r-a).
\]
This proves~\eqref{eq:Fbound}. On $[t,t+p]$ the integrand is
$w(r)-w(s)\geq w(r)-w(t+p)$, proving~\eqref{eq:increment}.
\end{proof}

Using the profiles and anchors fixed above, set
\begin{align}
 \kappa_S(\mu)&=(1-\eta)\sum_{T>S}F(r_S,r_T),\label{eq:C}\\
 Q_{f,b}(\mu)&=(1-\eta)\sum_{S\in R}w(r_S)g_*\mu_S
                    +\sum_{S\in R}\kappa_S(\mu)\delta_{h_S},\label{eq:Q}\\
 \tau(\mu)&=\sum_{S\in R}
 \bigl((1-\eta)w(r_S)\beta_S+\kappa_S(\mu)\bigr)\delta_{h_S}.
 \label{eq:tau}
\end{align}
At this stage these are finite positive sums of continuous valuations.
The first sum in~\eqref{eq:Q} is called the raw part; the second is
the buffer part, placing the additional mass $\kappa_S(\mu)$ at $h_S$.
The dependence of $\kappa_S$ and $\tau$ on the fixed choices $f,b$
is suppressed in their notation.
The next estimate verifies their subprobability budgets.

\begin{lemma}[Common capacity bound]\label{lem:capacity}
For every $\mu\in E$,
\begin{equation}\label{eq:capacity}
 Q_{f,b}(\mu)\leq\tau(\mu)\leq(1-\eta^2)\mu.
\end{equation}
\end{lemma}
\begin{proof}
First we show $g_*\mu_S\leq\beta_S\delta_{h_S}$ by evaluating on a
Scott-open set $U$. If $h_S\in U$, then
$g_*\mu_S(U)\leq\mu_S(D)=\beta_S$. If $h_S\notin U$, no $x\in C_S$
can satisfy $g(x)\in U$: otherwise $g(x)\leq h_S$ from
\eqref{eq:anchorupper} and upperness of $U$ would imply $h_S\in U$.
Thus $C_S$ misses the open $g^{-1}(U)$, and Lemma~\ref{lem:cell} gives
$g_*\mu_S(U)=0$. This proves the claimed domination, and comparison
of the finite sums in~\eqref{eq:Q} and~\eqref{eq:tau} gives $Q\leq\tau$.
Fix an open $U$, put $J=\{S:h_S\in U\}$, and let
$W=\bigcup_{S\in J}\up S$ in $R$.
By~\eqref{eq:anchorlower}, $k^{-1}W\subseteq U$.
The contribution of the coefficients $(1-\eta)w(r_S)\beta_S$
to $\tau(\mu)(U)$ is at most $(1-\eta)\mu(k^{-1}W)$.
By Lemma~\ref{lem:integral},
\[
 \kappa_S(\mu)\leq(1-\eta)nb r_S.
\]
For each $S\in J$, $r_S\leq\mu(k^{-1}W)$, and $|J|\leq n$.
The buffer part on $U$ is therefore at most
\[
 (1-\eta)nb\sum_{S\in J}r_S
 \leq(1-\eta)n^2b\mu(k^{-1}W)
 =(1-\eta)\eta\mu(k^{-1}W).
\]
Adding both parts and using $(1-\eta)(1+\eta)=1-\eta^2$
proves~\eqref{eq:capacity} on every open $U$.
In particular both outputs have total mass at most one.
\end{proof}

We next prove that $Q_{f,b}$ is Scott continuous. The proof has three
stages: monotonicity on simple valuations; ordinary real convergence
along directed input nets, which extends monotonicity to all valuations;
and, finally, preservation of directed suprema. This order matters because
the individual components $\mu_S$ need not be monotone in the input $\mu$.
Throughout, $f$ and $b$ remain fixed and we write $Q=Q_{f,b}$.
Only the input valuation varies; the profiles, anchors and functions
$w,F$ remain fixed.

\begin{lemma}\label{lem:simplemono}
The map $Q=Q_{f,b}$ is monotone on simple valuations.
\end{lemma}
\begin{proof}
Adding a positive atom increases every $r_S$, every old coefficient
$w(r_S)$, and every $\kappa_S$. It also adds a positive raw contribution.
Thus $Q$ increases under additions that keep total mass at most one.

Consider one upward move
$\mu'=\mu-p\delta_x+p\delta_y$, where $p\geq0$, $x\leq y$, and the removed
mass is available. Put $S=k(x)$ and $T=k(y)$, so $S\leq T$.
For each profile $V\in R$, write $r_V=\mu(A_V)$ and
$r'_V=\mu'(A_V)$; the prime marks evaluation at the valuation after
the move.
If $S=T$, all profile masses and buffers stay unchanged, and
$g(x)\leq g(y)$ proves $Q(\mu)\leq Q(\mu')$.

Suppose $S<T$. Every cumulative mass $r_V$ is nondecreasing, while
\[
 r'_S=r_S,\qquad r'_T=r_T+p\leq r_S.
\]
Put $\lambda=\mu-p\delta_x$, a positive simple valuation, and let
$Q_{\rm raw}$ denote the first sum in~\eqref{eq:Q}. Linearity of
the exact-cell restriction, with $\lambda_V$ defined
by~\eqref{eq:restriction} for the input $\lambda$, gives the following
identity, interpreted
on each Scott-open set:
\begin{equation}\label{eq:rawmove}
\begin{aligned}
 Q_{\rm raw}(\mu')-Q_{\rm raw}(\mu)
 &=(1-\eta)\sum_{V\in R}\bigl(w(r'_V)-w(r_V)\bigr)g_*\lambda_V\\
 &\quad +(1-\eta)p
 \bigl(w(r'_T)\delta_{g(y)}-w(r_S)\delta_{g(x)}\bigr).
\end{aligned}
\end{equation}
Every summand in the first line is nonnegative because $r'_V\geq r_V$
and $\lambda_V$ is a positive valuation.
Let $d=w(r_S)-w(r'_T)\geq0$. The summand indexed by $T$ in $\kappa_S$
increases, by~\eqref{eq:increment}, by at least $(1-\eta)pd$.
All other buffer increments are nonnegative. Since $g(x)\leq h_S$,
adding this increment at $h_S$ compensates for the second line
of~\eqref{eq:rawmove}: the relevant part is
\[
 (1-\eta)p\left[
 w(r'_T)(\delta_{g(y)}-\delta_{g(x)})
 +d(\delta_{h_S}-\delta_{g(x)})\right]\geq0.
\]
Here the final inequality holds after evaluation on every Scott-open set,
because such sets are upper.
Thus every upward atom move increases $Q$.
Lemma~\ref{lem:finiteflow} proves the result.
\end{proof}

\begin{lemma}\label{lem:ordinary}
For any directed $\mu_i\uparrow\mu$ and every Scott-open $U$,
\begin{equation}\label{eq:ordinary}
 Q(\mu_i)(U)\longrightarrow Q(\mu)(U)
\end{equation}
in ordinary real numbers. This conclusion does not assume global
monotonicity of $Q$.
\end{lemma}
\begin{proof}
For each $S\in R$ and Scott-open $V$, directed suprema in $E$ give
\[
\begin{aligned}
 \mu_i(V\cap A_S)&\uparrow\mu(V\cap A_S),\\
 \mu_i(V\cap B_S)&\uparrow\mu(V\cap B_S).
\end{aligned}
\]
Both limits are finite real numbers. Subtracting therefore yields
$\mu_{i,S}(V)\to\mu_S(V)$ in the ordinary real topology, even though
these differences need not increase. Taking $V=g^{-1}(U)$ gives
convergence of the raw valuation components in~\eqref{eq:Q}.
Also $r_S(\mu_i)=\mu_i(A_S)\uparrow\mu(A_S)=r_S(\mu)$.
Continuity of $w$ and $F$ now gives convergence of every weight and every
buffer coefficient $\kappa_S(\mu_i)$. Since $R$ is finite, the sums of
products in~\eqref{eq:Q} converge to $Q(\mu)(U)$.
This uses only real limits, without taking a directed supremum of the
outputs $Q(\mu_i)$.
\end{proof}

\begin{proposition}\label{prop:scott}
The map $Q_{f,b}:E\to E$ is Scott continuous.
\end{proposition}
\begin{proof}
First we prove global monotonicity. Let $\mu\leq\nu$ and fix a
Scott-open set $U$. Take a simple valuation $\lambda\ll\mu$.
Then $\lambda\ll\nu$. The valuation-domain theorem provides a directed
family of simple valuations $\nu_j\uparrow\nu$, so some $j_0$ satisfies
$\lambda\leq\nu_{j_0}$, and hence $\lambda\leq\nu_j$ for every
$j\geq j_0$. Both valuations in this comparison are simple.
Lemma~\ref{lem:simplemono} and then Lemma~\ref{lem:ordinary} give
\[
 Q(\lambda)(U)\leq\lim_j Q(\nu_j)(U)=Q(\nu)(U).
\]
Here we use only the elementary fact that an eventual non-strict
inequality is preserved by a convergent real net. We do not identify
this limit with a supremum of output valuations.

Next let $\lambda_i$ run through the directed simple valuations way
below $\mu$, with $\lambda_i\uparrow\mu$. The preceding inequality
holds for every $i$. Applying Lemma~\ref{lem:ordinary} a second time gives
\[
 Q(\mu)(U)=\lim_i Q(\lambda_i)(U)\leq Q(\nu)(U).
\]
Since $U$ was arbitrary, $Q(\mu)\leq Q(\nu)$. Thus global monotonicity
has been obtained using ordinary real limits, without presupposing it.

Finally let $\mu_i\uparrow\mu$ be any directed input net. The
monotonicity now proved makes $Q(\mu_i)$ an increasing directed family
in $E$. For each Scott-open $U$, its real evaluations are increasing,
so their ordinary limit equals their supremum. Equation~\eqref{eq:ordinary}
therefore gives
\[
 \Bigl(\sup_i Q(\mu_i)\Bigr)(U)
 =\sup_i Q(\mu_i)(U)=Q(\mu)(U).
\]
Directed suprema in $E$ are computed on opens, so
$\sup_i Q(\mu_i)=Q(\mu)$. Hence $Q$ is Scott continuous.
\end{proof}

We next prove finite separation. The capacity estimate in
Lemma~\ref{lem:capacity} leaves room to round the anchor coefficients
upward while keeping the resulting valuation below the input.
We will use a single rounding grid for all input valuations.

\begin{proposition}\label{prop:finiteSep}
Each map $Q_{f,b}$ is finitely separated.
\end{proposition}
\begin{proof}
Fix $f,b$ and write $Q=Q_{f,b}$ throughout this proof.
Write $\tau(\mu)=\sum_S d_S(\mu)\delta_{h_S}$ using the indexed
coefficients in~\eqref{eq:tau}. If $d_S(\mu)>0$, then $r_S>a$.
Indeed, either $w(r_S)>0$, or a positive $F(r_S,r_T)$ occurs;
both imply $r_S>a$.

Choose a positive integer $N$ so large that
\begin{equation}\label{eq:grid}
 \frac{n}{N}<\eta^2a.
\end{equation}
The same $N$ is used for every input $\mu$: it depends only on the
fixed parameters $n,\eta,a$.
Define the rounded coefficients by
\[
 \widehat d_S=\frac{\lceil Nd_S(\mu)\rceil}{N}.
\]
Here $\lceil t\rceil$ denotes the ceiling of $t$, namely the smallest
integer greater than or equal to $t$, and the hat marks the rounded
coefficient. Thus $\widehat d_S$ is the smallest integer multiple of
$1/N$ that is at least $d_S(\mu)$. In particular,
\[
 d_S(\mu)\leq\widehat d_S<d_S(\mu)+\frac1N,
 \qquad d_S(\mu)=0\ \Longrightarrow\ \widehat d_S=0.
\]
Put $\gamma_\mu=\sum_S\widehat d_S\delta_{h_S}$.
Clearly $Q(\mu)\leq\tau(\mu)\leq\gamma_\mu$.

For an open $U$ with $\tau(\mu)(U)=0$, rounding creates no mass
on $U$. Otherwise some $h_S\in U$ has $d_S>0$, and
\eqref{eq:anchorlower} gives $\mu(U)\geq r_S>a$.
The capacity bound and~\eqref{eq:grid} yield
\[
 \gamma_\mu(U)
 \leq\tau(\mu)(U)+\frac nN
 \leq(1-\eta^2)\mu(U)+\frac nN
 \leq\mu(U).
\]
Thus $\gamma_\mu\leq\mu$, including the total-mass test $U=D$.
Every such witness belongs to the fixed finite set
\begin{equation}\label{eq:Gamma}
 \Gamma=\left\{\sum_{S\in R}\frac{j_S}{N}\delta_{h_S}:
 j_S\in\mathbb N_0,\quad\sum_{S\in R}j_S\leq N\right\}.
\end{equation}
Repeated anchors do not affect finiteness or the inequalities;
the indexed coefficients may be merged afterwards.
We have $Q(\mu)\leq\gamma_\mu\leq\mu$ with
$\gamma_\mu\in\Gamma$, proving finite separation.
\end{proof}

\section{Proof of the main theorem}\label{sec:cofinal}

We now combine the preceding approximants into a directed FS approximate
identity. By Lemma~\ref{lem:saturation}, it suffices to show that one
approximant can meet any finite list of approximation requirements.

For this purpose, assume that $D$ is an FS-domain and fix an FS approximate
identity $(f_i)_{i\in I}$. Its index set $I$ is directed, with
$i\leq i'$ implying $f_i\leq f_{i'}$.
For each $i$, choose a finite separator $M_i$ and representatives as
in Section~\ref{sec:construction}; write $k_i$ for the resulting profile map,
$R_i=k_i[D]$ for its image, and $n_i=|R_i|$.
These choices are made once for each $i$, independently of $b$ and
of the input valuations.

\begin{lemma}\label{lem:scaled}
For this approximate identity, the Scott-continuous maps
\[
 R_{i,t}:E\longrightarrow E,\qquad
 R_{i,t}(\mu)=t\,(f_i^2)_*\mu,\qquad i\in I,\quad 0<t<1,
\]
are directed and have pointwise supremum $\id_E$.
Here $t$ multiplies the mass of the pushforward valuation.
The indices are ordered by
$(i,t)\leq(i',t')$ if $i\leq i'$ and $t\leq t'$.
\end{lemma}
\begin{proof}
The family $(f_i^2)$ is directed and has supremum $\id_D$.
For the latter assertion, Scott continuity gives
\[
 \sup_{i\in I}\sup_{j\in I}f_i(f_j(x))
 =\sup_{i\in I}f_i(x)=x.
\]
For any pair $i,j$, choose $k\geq i,j$. Pointwise order and
monotonicity give $f_i(f_j(x))\leq f_k(f_k(x))=f_k^2(x)$.
Thus the diagonal pairs $(k,k)$ dominate all pairs $(i,j)$,
which is why the squares have the same supremum.
For each open $U$, the opens $(f_i^2)^{-1}U$ increase with union $U$.
Continuous valuations therefore give $(f_i^2)_*\mu\uparrow\mu$.
Finally, as $t\to1^-$, that is, as $t$ approaches $1$ from below,
the scalar factors tend to $1$. Directedness follows from pointwise
order of the underlying maps and from increasing nonnegative scalars.
\end{proof}

\begin{proposition}\label{prop:batch}
The maps $Q_{f_i,b}$, with $i\in I$ and $0<b<n_i^{-2}$,
satisfy the finite-batch hypothesis of
Lemma~\ref{lem:saturation} on $E$.
\end{proposition}
\begin{proof}
Let $\theta_j\ll\mu_j$ for $j=1,\ldots,m$ be finitely many tests.
For each $j$, the way-below relation and Lemma~\ref{lem:scaled}
give a pair $(i_j,t_j)$ with $\theta_j\leq R_{i_j,t_j}(\mu_j)$.
Directedness supplies one pair $(i,t)$ above all these finitely many
pairs. Thus $0<t<1$ and
\[
 \theta_j\leq t(f_i^2)_*\mu_j\quad\text{for every }j.
\]
Keep this index $i$, the scalar $t$, and all inputs $\mu_j$ fixed.
For readability write $R=R_i$ and $n=n_i$ in the rest of this proof.
If every $\beta_S(\mu_j)$ is zero, then all inputs and targets are
zero, and any $b\in(0,n^{-2})$ gives the required map.
Otherwise define
\[
 r_{\min}=\min\{r_S(\mu_j):1\leq j\leq m,\ S\in R,
                         \ \beta_S(\mu_j)>0\}.
\]
This is the minimum of a nonempty finite collection of positive
numbers: $r_S(\mu_j)\geq\beta_S(\mu_j)>0$ for every entry in the
collection. Hence $r_{\min}>0$. It depends on the fixed profiles and
inputs, and is independent of $b$.

We now vary only $b$. Recall from~\eqref{eq:wF} that
\[
 w_b(r_{\min})=\max\{0,1+b\log r_{\min}\}.
\]
Since $\log r_{\min}$ is a fixed finite real number,
$b\log r_{\min}\to0$ as $b\to0^+$.
Here $b\to0^+$ means that $b$ approaches zero through positive values.
Consequently
\[
 w_b(r_{\min})\to1,\qquad 1-n^2b\to1
 \qquad\text{as }b\to0^+,
\]
and therefore
\[
 (1-n^2b)w_b(r_{\min})\to1
 \qquad\text{as }b\to0^+.
\]
Because the chosen $t$ satisfies $t<1$, we can take one sufficiently
small $b\in(0,n^{-2})$ such that
$(1-n^2b)w_b(r_{\min})\geq t$.
For every cell with $\beta_S(\mu_j)>0$, the inequality
$r_S(\mu_j)\geq r_{\min}$ and monotonicity of $w_b$ then give
\[
 (1-n^2b)w_b(r_S(\mu_j))
 \geq(1-n^2b)w_b(r_{\min})\geq t.
\]
The left-hand side is exactly the coefficient of
$(f_i^2)_*(\mu_j)_S$ in the raw part of~\eqref{eq:Q}.
Cells with $\beta_S(\mu_j)=0$ contribute the zero valuation.
The buffer contributions are nonnegative, so
\[
 Q_{f_i,b}(\mu_j)\geq t\sum_S(f_i^2)_*(\mu_j)_S
 =t(f_i^2)_*\mu_j\geq\theta_j
\]
simultaneously for every $j$. Here $(\mu_j)_S$ denotes the exact-cell
component of $\mu_j$ defined in~\eqref{eq:restriction}.
\end{proof}

\begin{proof}[Proof of Theorem~\ref{thm:main}]
The valuation-domain theorem makes $E=\Vsub(D)$ continuous.
Propositions~\ref{prop:scott}, \ref{prop:finiteSep} and
\ref{prop:batch}, followed by Lemma~\ref{lem:saturation}, give the
directed FS approximate identity
\begin{equation}\label{eq:finalFS}
 \left\{Q_{f_i,b}^{\,2}:\ i\in I,\quad 0<b<n_i^{-2}\right\}.
\end{equation}
Representatives are fixed for each $i$; the saturation argument applies
to the resulting set of maps. In particular, directedness is proved
for the square family, and is not assumed for the logarithmic parameters.
This proves the subprobability assertion.

Let $\bot$ be the least point of $D$. Define the normalization map
$N:E\to\Vone(D)$ by
\[
 N(\mu)=\mu+(1-\mu(D))\delta_\bot\in\Vone(D).
\]
It places the missing mass $1-\mu(D)$ at the least point.
On a proper Scott open, $N(\mu)$ agrees with $\mu$; on $D$ its
value is one. These evaluations show directly that $N$ is monotone
and Scott. The inclusion $j:\Vone(D)\to\Vsub(D)$ is Scott and
$Nj=\id_{\Vone(D)}$. FS-domains are closed under Scott retractions:
if $e:Y\to X$ and $p:X\to Y$ are Scott-continuous maps with
$pe=\id_Y$, and $q:X\to X$ has finite separating set $M$, then $pqe$ is
separated by $p[M]$, and a directed approximate identity is preserved
by this composition. Applying this closure property to $N,j$ proves
that $\Vone(D)$ is FS.
\end{proof}

\begin{flushleft}
\small
\textsc{Xiaoyong Xi}\\
School of Mathematics and Statistics, Jiangsu Normal University,
Jiangsu, China.\\
\textit{Email:} \href{mailto:xixy@yctu.edu.cn}{xixy@yctu.edu.cn}

\medskip
\textsc{Chong Shen}\\
School of Mathematical Sciences, Beijing University of Posts and
Telecommunications, Beijing, China. Key Laboratory of Mathematics and
Information Networks (Beijing University of Posts and Telecommunications),
Ministry of Education.\\
\textit{Email:} \href{mailto:shenchong0520@163.com}{shenchong0520@163.com}

\medskip
\textsc{Dongsheng Zhao}\\
Mathematics and Mathematics Education, National Institute of Education,
Nanyang Technological University, 1 Nanyang Walk, Singapore 637616.\\
\textit{Email:} \href{mailto:dongsheng.zhao@nie.edu.sg}{dongsheng.zhao@nie.edu.sg}
\end{flushleft}

\begin{thebibliography}{9}

\bibitem{ChenKouLyu}
Y. Chen, H. Kou and Z. Lyu,
\emph{Finite-valuation approximable structures: a solution to the
Jung--Tix problem of probabilistic powerdomains},
arXiv:2608.03073v3, 2026.
\href{https://arxiv.org/abs/2608.03073v3}{arXiv:2608.03073v3}.

\bibitem{Jones}
C. Jones,
\emph{Probabilistic Non-Determinism},
PhD thesis, University of Edinburgh, 1990;
LFCS report ECS-LFCS-90-105.
\href{https://hdl.handle.net/1842/413}{hdl:1842/413}.

\bibitem{JungTix}
A. Jung and R. Tix,
The troublesome probabilistic powerdomain,
\emph{Electronic Notes in Theoretical Computer Science}
\textbf{13} (1998), 70--91.
\href{https://doi.org/10.1016/S1571-0661(05)80216-6}
{doi:10.1016/S1571-0661(05)80216-6}.

\end{thebibliography}
\end{document}